\documentclass[11pt,a4paper]{amsart}[2010]
\usepackage{tikz}
\usepackage{enumitem}
\title{A rigid hyperfinite type $\mathrm{II}_1$ factor}

\author{Ilijas Farah}

\address{Department of Mathematics and Statistics,
York University,
4700 Keele Street,
Toronto, Ontario, Canada, M3J
1P3} 
\address{Matemati\v cki Institut SANU\\
Kneza Mihaila 36\\
11\,000 Beograd, p.p. 367\\
Serbia}
\email{ifarah@yorku.ca}
\urladdr{http://www.math.yorku.ca/$\sim$ifarah}

\author{Ilan Hirshberg}

\address{Department of Mathematics, Ben Gurion University of the Negev, \phantom{----------------}\linebreak\text{}\hspace{3.5mm}
P.O.B. 653, Be'er Sheva 84105, Israel}
\email{ilan@math.bgu.ac.il}

\thanks{This research was supported by Israel Science Foundation grant 476/16 and by the Fields Institute (I.H.) and NSERC (I.F.).}

\usepackage{amsmath}
\usepackage{amssymb}
\usepackage{amsthm}

\allowdisplaybreaks

\theoremstyle{plain}

\newtheorem{Thm}{Theorem}%[section]

\newtheorem{Lemma}[Thm]{Lemma}

\newtheorem{Prop}[Thm]{Proposition}

\theoremstyle{definition}

\newtheorem{Rmk}[Thm]{Remark}

\newcommand{\N}{{\mathbb N}}
\newcommand{\Z}{{\mathbb Z}}

\newcommand{\Q}{{\mathbb Q}}

\newcommand{\cU}{{\mathcal U}}
\newcommand{\cV}{{\mathcal V}}
\newcommand{\cF}{{\mathcal F}}

\newcommand{\cO}{\mathcal{O}}

\newcommand{\ant}{\mathrm{Ant}}
\newcommand{\aut}{\mathrm{Aut}}
\newcommand{\inn}{\mathrm{Inn}}
\newcommand{\out}{\mathrm{Out}}

\newcommand{\eps}{\varepsilon}
\numberwithin{equation}{section}

\newcommand{\id}{\mathrm{id}}
\newcommand{\Ad}{\mathrm{Ad}}

\newcommand{\Sym}{\mathrm{Sym}}

\newcommand{\cstar}{$\mathrm{C}^*$}
\newcommand{\cst}{\mathrm{C}^*}
\newcommand{\doo}{\diamondsuit_{\aleph_1}}

\newcommand{\twoone}{$\mathrm{II}_1$}

\DeclareMathOperator{\Struct}{Struct}

\newcommand{\fA}{\mathfrak A}
\newcommand{\fB}{\mathfrak B}
\newcommand{\rs}{\upharpoonright}
\newcommand{\calL}{\mathcal L}

\newcommand{\sfS}{\mathsf S}
\newcommand{\sfC}{\mathsf C}
\newcommand{\bbX}{\mathbb X}
\DeclareMathOperator{\Xm}{\bbX}
\DeclareMathOperator{\XmR}{\bbX_{R}}

\DeclareMathOperator{\Code}{Code}

\newcounter{my_enumerate_counter}
\newcommand{\pushcounter}{\setcounter{my_enumerate_counter}{\value{enumi}}}
\newcommand{\popcounter}{\setcounter{enumi}{\value{my_enumerate_counter}}}

\begin{document}
\begin{abstract}
We show that it is relatively consistent with ZFC that there exists a hyperfinite type $\mathrm{II}_1$-factor of density character $\aleph_1$ which is not isomorphic to its opposite, does not have any outer automorphisms, and has trivial fundamental group.

\end{abstract}     

\maketitle

The goal of this paper is to prove the following theorem.

\begin{Thm}\label{Thm:Main} It is relatively consistent with ZFC that there exists a  \twoone{} factor that is not isomorphic to its opposite, has no outer automorphisms, has trivial fundamental group, is hyperfinite, and is of density character $\aleph_1$. 
Also, the fundamental group of its ultrapower associated with any nonprincipal ultrafilter on $\N$ is equal to $(0,\infty)$.  
\end{Thm}

Type $\mathrm{II}_1$ and type $\mathrm{III}$ factors with separable preduals which are not isomorphic to their opposites were constructed by Connes in \cite{Connes-not-anti-isomorphic,Connes-comptes-rendus}. An example of a type $\mathrm{II}_1$ factor with separable predual which has has no outer automorphisms and has trivial fundamental group was constructed in \cite{Ioana-Peterson-Popa}. Those examples of course are not hyperfinite. Theorem~\ref{Thm:Main} provides a factor constructed as a transfinite inductive limit of copies of the hyperfinite $\mathrm{II}_1$ factor which exhibits those properties.

Another curious property of the \twoone{} factor constructed in Theorem~\ref{Thm:Main} is that the fundamental group of its ultrapower is strictly larger than the closure of its fundamental group. To the best of our knowledge, this is the first example of a \twoone{} factor with this property; it is not known whether a factor with separable predual can have this property.  

We prove that the conclusion of Theorem~\ref{Thm:Main} follows from Jensen's~$\doo$. This axiom was first applied to  operator algebras in \cite{AkeWe:Consistency} in order to construct a counterexample to Naimark's problem. The Akemann--Weaver construction was subsequently adapted in \cite{farah2016simple} to construct a simple nuclear \cstar-algebra which is not isomorphic to its opposite. Those techniques were further refined in \cite{vaccaro2017trace}. While the set theoretic machinery we use here is similar to  the one used in those papers (albeit somewhat simplified), the operator algebraic techniques turn out to be very different in nature.  The results in \cite{AkeWe:Consistency,farah2016simple,vaccaro2017trace} rely on studying the action of outer automorphisms and antiautomorphisms on the pure states of a separable \cstar-algebra, and use in an essential way results due to Kishimoto (\cite{kishimoto-1981}) and work of 
 Kishimoto, Ozawa, and Sakai (\cite{KiOzSa}) about the homogeneity of the pure state space of separable \cstar-algebras. Beyond the fact that pure states of von Neumann algebras are generally not normal, the homogeneity result of Kishimoto--Ozawa--Sakai breaks down for non-separable \cstar-algebras, and in particular for type $\mathrm{II}_1$-factors (\cite[Remark 2.3]{KiOzSa}).
  Theorem~\ref{Thm:Main} is the first application of $\doo$ to von Neumann algebras, and it answers the question stated in~\cite[Remark~3.3]{farah2016simple}. Nonisomorphic hyperfinite \twoone{}factors with preduals of the same uncountable density character were first constructed in \cite{Wid:Nonisomorphic}. In \cite[Theorem~1.3]{FaKa:NonseparableII} it was proved that for every uncountable cardinal $\kappa$ there are $2^\kappa$ nonisomorphic hyperfinite \twoone{} factors with predual of density character $\kappa$.  In spite of being nonisomorphic, all of these factors constructed in a similar manner and they are unlikely to have any of the properties of the factor constructed in Theorem~\ref{Thm:Main}. 

We briefly outline the idea of the construction. We construct our factor~$M$ as a transfinite inductive limit of copies of the hyperfinite \twoone~ factor $R$, indexed as  $R_{\xi}$ for  countable ordinals $\xi$. Suppose we want to make sure that~$M$ does not have any outer automorphisms. (The idea for eliminating antiautomorphisms and isomorphisms into corners is similar.) For any automorphisms of $M$ there exist ``many'' ordinals $\xi$ such that the automorphism leaves $R_{\xi}$ invariant and the restriction is outer. Our inductive step does this: given an outer automorphism $\beta$ of $R_\xi$ for some $\xi$, we find a way to embed $R_\xi$ into a larger copy of $R$, denoted $R_{\xi+1}$, which has the property that $\beta$ cannot extend not only to $R_{\xi+1}$, but in fact to any larger hyperfinite \twoone~factor containing $R_{\xi+1}$. This allows us at each step to kill off a possible restriction of an automorphism of the yet-to-be constructed inductive limit. To make sure that we eliminate all possible outer automorphisms of $M$, we need a prediction device which should tell us which automorphism to handle at each stage; for that we use Jensen's $\diamondsuit_{\aleph_1}$ axiom, which is known to be relatively consistent with ZFC.

\subsection*{Acknowledgments} We are indebted to Sorin Popa for remarks on the original draft of this paper. We also thank the referee for some helpful suggestions.

\section{The obstructions}

 In this section we describe the device used to create obstructions to extending outer automorphisms,  antiautomorphisms, and isomorphisms onto a corner associated to a projection of trace not equal to 1 of subfactors of the \twoone{} factor.  This  is used in the proof of Theorem~\ref{Thm:Main}. 

For a type \twoone{} factor $M$, we denote  the set of all anti\-auto\-mor\-phisms of $M$ by $\ant(M)$. Note that $\ant(M) \cup \aut(M)$ is a group.

Let $G$ be a group, and let $g,h \in G$. Let $a$ and $b$ be the standard generators of $F_2 = \Z * \Z$. By $\pi_{g,h} \colon \Z * \Z \to G$ we denote the canonical homomorphism which satisfies $\pi_{g,h}(a)=g$ and $\pi_{g,h}(b)=h$.   Notice that if $\alpha \in \aut(M)$ and $\beta \in \ant(M) \cup \aut(M)$ then for any $w \in \left < a , b^{-1}ab \right >$, we have $\pi_{\alpha,\beta}(w) \in \aut(M)$. As usual,  by $\inn(M)$ we denote the group of  inner automorphisms of $M$, and $\out(M) = \aut(M)/\inn(M)$. The group $\aut(M)$ is a topological group when endowed with the point-$\|\cdot\|_2$-topology, that is, the topology which is generated by open sets of the form:
\[
\cO_{F,\eps,\alpha} = \{\varphi \in \aut(M) \mid  \forall a \in F,  \|\varphi(a) - \alpha(a) \|_2 + \|\varphi^{-1}(a) - \alpha^{-1}(a) \|_2 < \eps   \}
\]
for $\alpha \in \aut(M)$,  a finite set of contractions $F \subset M$, and $\eps>0$. If $M$ has separable predual then $\aut(M)$ is a Polish group with this topology (see e.g., \cite[\S 7.5.2]{anantharaman2017introduction}, where it was observed that $\aut(M)$ is isomorphic to a closed subgroup of the unitary group of  $L^2(M,\tau)$ equipped with the strong operator topology).
When $M$ is the hyperfinite \twoone{} factor $R$, it follows from the classification of automorphisms of $R$ from \cite{Connes-1975} that  the inner automorphisms are dense in $\aut(M)$.

Our first goal is to prove Theorem \ref{Thm:free group}. We need a few lemmas; the first of which, for the case of automorphisms, is an immediate application of the Connes' Rokhlin-type theorem. Fix a free ultrafilter $\cV$. By $R^{\cV}$ we denote the tracial ultrapower,  $l^{\infty}(\N,R) /\{ f \in l^{\infty}(\N,R) \mid \lim_{n \to \cV}\|f(n)\|_2 = 0 \}$. If $\alpha$ is an automorphism of $R$, by abuse of notation, we use $\alpha$ to denote both the induced automorphism of $R^{\cV}$ and of the central sequence algebra $R^{\cV} \cap R'$. 

\begin{Lemma}
	\label{Lemma:Rokhlin projections}
	Suppose $\beta$ is either an outer automorphism or an anti\-au\-to\-mor\-phism of $R$. Then there exist orthogonal projections $p_0,p_1 \in R^{\cV} \cap R'$ such that $\tau(p_0) = \tau(p_1) \geq 1/3$ and $\beta(p_0) = p_1$. 
\end{Lemma}
\begin{proof}
	If $\beta$ is an outer automorphism and has infinite order in $\out(R)$, this follows from the Connes' Rokhlin-type theorem, \cite[Chapter XVII, Lemma 2.3]{Takesaki-III}, where we pick $n=2$. If $\beta$ has finite order in $\out(R)$, then by \cite[Chapter XVII, Theorem 2.10]{Takesaki-III}, the automorphism $\beta$ is cocycle conjugate to an automorphism of the form $\beta \otimes \sigma_{p}$, where $p$ is the period of $\beta$ in
	$\out(R)$ and $\sigma_p$ is an infinite tensor product action on $\overline{\bigotimes}_{1}^{\infty}M_p$ of cyclic permutations. This has a central sequence of projections which are permuted cyclically. Therefore, there exist projections $q_0,q_1,\ldots,q_{p-1} \in R^{\cV}\cap R'$ such that $\beta (q_j) = q_{j+1\mod p}$ for all $j$. If $p$ is even, set $p_0 = q_0 + q_2 + \ldots + q_{p-2}$, and if $p$ is odd then set $p_0 = q_0 + q_2 + \ldots q_{p-3}$, and set $p_1 = \beta(p_0)$. If $p$ is even then $\tau(p_0)=\tau(p_1) = 1/2$, and if $p$ is odd then $\tau(p_0) = \tau(p_1) = 1/2 - 1/2p \geq 1/3$, as required. 
	
	If $\beta$ is an antiautomorphism, by \cite[Lemma 2.1]{Giordano}, up to conjugation by an automorphism, for any $n \in \N$ there exists a unital copy of $M_n \subset R^{\cV} \cap R'$ such that $\alpha|_{M_n}$ is given by the transpose map, that is, $\alpha|_{M_n}(e_{jk}) = e_{kj}$ for the standard matrix units $\{e_{jk}\}_{k,j=1,2,\ldots,n}$. Set $n=2$, then the projections
		\[
		p_0 = \left ( 
				\begin{matrix}
				1/2 & i/2 \\	
				-i/2 & 1/2
				\end{matrix}
			\right ) 
			\quad , \quad
		p_1 = \left ( 
					\begin{matrix}
					1/2 & -i/2 \\	
					i/2 & 1/2
					\end{matrix}
				\right ) 
		\]
		satisfy the requirements.
\end{proof}

\begin{Lemma}
	\label{Lemma:mixing up Rokhlin projections}
	 Let $\beta$ be an outer automorphism or an antiautomorphism of~$R$.  Let $w \in \left < a , b^{-1}ab \right >$ be a nontrivial element. Then there exist a unitary $u \in U(R^{\cV} \cap R')$ and a projection $p \in R^{\cV} \cap R'$ with $\tau(p)\geq 1/3$ such that $\pi_{\Ad(u),\beta}(w)(p) \perp p$.
\end{Lemma}
	\begin{proof}
		The canonical homomorphism $\Z \to \Z_2$ gives rise to a homomorphism $\varphi \colon \left < a , b \right > \to \Z * \Z_2$, such that $\varphi|_{\left < a , b^{-1}ab \right >}$ is injective. Since $\Z * Z_2$ is residually finite, we can pick a homomorphism  $\psi \colon \Z * \Z_2 \to \Sym(S)$ into a symmetry group of a finite set $S$, such that any nontrivial element in the image has no fixed points and such that $\psi(\varphi(w)) \neq 1$. Set $\sigma_{a} = \psi \circ \varphi(a)$, $\sigma_b =  \psi \circ \varphi(b)$ and  $\sigma_w = \psi \circ \varphi(w)$. Since $\sigma_b$ is of order 2 with no fixed points, we can decompose $S$ into a disjoint union $S_0 \sqcup S_1$ where $\sigma_b|_{S_0} \colon S_0 \to S_1$ is a bijection. Now, let $p_0,p_1$ be projections as in Lemma \ref{Lemma:Rokhlin projections}. Since $ R^{\cV} \cap R'$ is a \twoone~ factor (see \cite[Chapter XIV, Theorem 4.18]{Takesaki-III}), we can decompose $p_0$ into a direct sum of equivalent orthogonal projections $p_0 = \bigoplus_{s \in S_0} q_s$. For $s \in S_1$, set $q_s = \beta(q_{\sigma_b^{-1}(s)})$. Then $q_s \leq p_1$ and we have $p_1 = \bigoplus_{s \in S_1} q_s$. Since the projections $\{q_s\}_{s \in S}$ are pairwise equivalent and orthogonal and $R^\cV\cap R'$ is a factor, there exists a unitary $u \in R^{\cV} \cap R'$ such that $\Ad(u)(q_s) = q_{\sigma_a(s)}$ for all $s \in S$. Thus, for any $s \in S$, $\pi_{\Ad(u),\beta}(w)(q_s) = q_{\sigma_w(s)}$. Since the permutation $\sigma_w$ has no fixed points, and $S$ is even, we can find a subset $S' \subset S$ consisting of half the points such that $S = S' \sqcup \sigma_w(S')$. Set $p = \sum_{s \in S'}q_s$, then $\pi_{\Ad(u),\beta}(w)(p) \perp p$, and $\tau(p) = \tau(p_0) \geq 1/3$, as required.
	\end{proof}

%%%%%!!!

\begin{Lemma} 
	\label{Lemma:making-automorphism-free}
	Let  $\beta$ be an outer automorphism or an antiautomorphism of $R$.
	Let $w \in \left <a , b^{-1}ab \right > < \left < a,b \right >$ be a nontrivial element. 
	  Then for any $\delta>0$, for any finite set of unitaries $\cU_0 \subset U(R)$ and for any nonempty open   $\cO\subseteq U(R)$  there exist an automorphism $\alpha' \in \cO$ and a projection $p \in R$ such that for every $z \in \cU_0$ we have $\| \pi_{\alpha,\beta}(w) (p) - zpz^*\|_2^2>2/3 - \delta$. 
\end{Lemma}
\begin{proof}
Since  approximately inner automorphisms are dense in $\aut(R)$, we may assume that for a finite set of contractions $F$, a given $\eps>0$, and a given $v\in U(R)$, the open set $\cO$ is of the form 
\[
\cO_{F,\eps,v} = \{\varphi \in \aut(R) \mid  (\forall a \in F)  \|\varphi(a) - \Ad v(a) \|_2 +  \|\varphi^{-1}(a) - \Ad v^*(a) \|_2 < \eps   \}
\]
 By Lemma \ref{Lemma:mixing up Rokhlin projections} there  are a unitary $u$  and a projection $p$ in $R^{\cV} \cap R'$ 
such that  $\tau(p)\geq 1/3$ and $\pi_{\Ad(u),\beta}(w)(p) \perp p$.  Lift $u$ to a sequence $(u_1,u_2,\ldots)$ in $U(l^{\infty}(R))$, and lift $p$ to a sequence of projections $(p_1,p_2,\ldots)$ in  $l^{\infty}(R)$. Set $\gamma_n = \Ad(u_n) \circ \Ad(v)$. Notice that automorphism $\gamma = (\gamma_1,\gamma_2,\ldots)$ of $l^{\infty}(R)$ descends to the automorphism $\Ad(u)$ of $R^{\cV} \cap R'$. Thus, for any $z \in \cU_0$ we have $\|zp_nz^* - p_n\|_2 \to 0$ and $\lim_{n \to \cV} \|\pi_{\gamma_n,\beta}(w)(p_n) \cdot p_n\|_2=0$.  Therefore 
\[
\lim_{n \to \cV} \|\pi_{\gamma_n,\beta}(w)(p_n) - zp_nz^*\|_2^2 = 2\tau(p) \geq 2/3.
\]
 Furthermore, since $\lim_{n \to \cV} \| [u_n,x] \|_2 = \lim_{n \to \cV} \| [u_n^*,x] \|_2 = 0$ for any $x \in F \cup \{v\}$, we have 
 \[
 \lim_{n \to \cV} \|\gamma_n(x) - \Ad(v)(x)\|_2 = \lim_{n \to \cV} \|\gamma_n^{-1}(x) - \Ad(v^*)(x)\|_2 = 0
 \]
 for all $x \in F$; in particular, the set of $n$ such that  $\gamma_n \in \cO_{F,\eps,v}$ belongs to $\cV$. 
Thus, we can pick an index $n$ such that for all $z\in \cU_0$ we have $\|\pi_{\gamma_n,\beta}(w)(p_n) - zp_nz^*\|_2^2 \geq 2/3 - \delta$ and $\gamma_n \in \cO$; we now set $p = p_n$, $\alpha' = \gamma_n$ and we are done.
\end{proof}

The following result is similar in spirit to the freeness results in \cite[Lemma~A.2]{Ioana-Peterson-Popa} (for outer automorphisms) and \cite[Remark~A.3 (2)]{Ioana-Peterson-Popa} (for antiautomorphisms), but it does not obviously follow from them (and they don't follow from our result).  

\begin{Thm}
	\label{Thm:free group}
	Let $\beta$ be an outer automorphism or an antiautomorphism of the hyperfinite $\mathrm{II}_1$ factor. Then there exists $\alpha \in \aut (R)$ such that the images of the automorphisms $\alpha$ and $\beta^{-1} \circ \alpha \circ \beta$ in $\out(R)$ generate a free group.
\end{Thm}
\begin{proof}
	 Let $\cU = \{u_1,u_2,\ldots\}$ be a dense sequence of unitaries in $U(R)$ (in the strong operator topology). Note that $\cU$ spans a SOT-dense subset of $R$. Let $\cU_n = \{u_1,u_2,\ldots,u_n\}$. Let $w \in \left < a,b^{-1}ab  \right > \smallsetminus \{1\}$. Denote by $\mathcal{P}(R)$ the set of all projections in $R$. Define:
	\[
	\cO(w,n) = \{\gamma \in \aut(R) \mid (\exists p \in \mathcal{P}(R))( \forall z \in \cU_n) \|\pi_{\gamma,\beta}(w)(p) - zpz^*\|_2^2>1/2\}
	\]
	By Lemma \ref{Lemma:making-automorphism-free}, the set $\cO(w,n)$ is dense. It is also clearly open. As $\left < a,b^{-1}ab  \right >$ is countable, by the Baire Category Theorem, the set 
	\[
	G_0 = \bigcap_{w \in \left < a,b^{-1}ab  \right > \smallsetminus \{1\}} \bigcap_{n=1}^{\infty} \cO(w,n)
	\] 
	is dense. Pick any $\alpha \in G_0$. We claim that $\pi_{\alpha,\beta}(w)$ is not inner for any $w \in \left < a,b^{-1}ab  \right > \smallsetminus \{1\}$. 
	Indeed, if there exists such a word $w$ and a unitary $u \in U(R)$ which implements $\pi_{\alpha,\beta}(w)$, then we can  pick $u_n \in \cU$ such that $\|u_n - u\|_2 < 1/4$ and  a projection $p$ such that $\|\pi_{\alpha,\beta}(w)(p) - u_npu_n^*\|_2^2>1/2$. Therefore  $\|upu^*-u_npu_n^*\|_2 < 1/2$ and
	\[
	\|\pi_{\alpha,\beta}(w)(p) - u p u^*\|_2 > 1/\sqrt{2} - 1/2 > 0,
	\] so finally $\pi_{\alpha,\beta} \neq \Ad(u)$.
\end{proof}

The following Proposition is based on \cite[Lemma on p.~245]{wassermann1976tensor}. %It remains true (with the same proof) when $F_2$ is replaced by any discrete group. 

\begin{Prop}\label{P:Wassermann}
	Let $M$ be a von Neumann algebra with a faithful trace $\tau$. Suppose $u,v \in U(M)$ are such that for every $w \in \left < a,b \right > \smallsetminus \{1\}$ we have $\tau(\pi_{u,v}(w)) = 0$. Then $\cst(u,v) \cong \cst_r(F_2)$. 
\end{Prop}

\begin{proof}
	We view $M$ as represented on $L^2(M,\tau)$ via the standard representation, that is, the GNS representation associated to $\tau$. Let $\xi \in L^2(M,\tau)$ be the GNS vector, so that $\left < x\xi,\xi \right > = \tau(a)$ for all $a \in M$. For any word $w$ in the free group on two generators, set $\xi_w = \pi_{u,v}(w)\xi$. If $w_1 \neq w_2$, then 
	\begin{equation*}
	\begin{split}
	\left < \xi_{w_1} , \xi_{w_2} \right > & = \left < \pi_{u,v}(w_1)\xi , \pi_{u,v}(w_2)\xi \right >  = \left < \pi_{u,v}(w_2)^* \pi_{u,v}(w_1)\xi , \xi \right > 
	\\
	& =  \left < \pi_{u,v}(w_2^{-1}w_1)\xi , \xi \right >  = \tau(\pi_{u,v}(w_2^{-1}w_1)) = 0 .
	\end{split}
	\end{equation*}
	Thus, $H = \overline{ \mathrm{span}\{\xi_w \mid w \in F_2\} } \cong l^2(F_2)$ is invariant for $\cst(u,v)$, and the action of the group generated by $u,v$ on $H$ is unitarily equivalent to the left regular representation. Let $P_H$ be the projection onto $H$, then the map $\varphi \colon \cst(u,v) \to \cst(u,v)P_H$ is a quotient. If $x \in \ker(\varphi)$ then $x P_H = 0$, and in particular $x \xi = 0$. Since $\xi$ is a separating vector, it follows that $x = 0$. Therefore, $\varphi$ is in fact an isomorphism.  
	So, for any $*$-polynomial $p$ in $u,v$, we have $\|p(u,v)\| = \|p(u,v) P_H\| = \|p(u,v)\|_{\cst_r(F_2)}$, so $\cst(u,v) \cong \cst_r(F_2)$, as required.
\end{proof}

We are now ready to prove the first of the two main results of this section, used as steps in the proof of Theorem \ref{Thm:Main}.

\begin{Thm}\label{P:extend} 
		Suppose $\beta$ is either an outer automorphism or an anti\-au\-to\-mor\-phism of $R$.  Then there exists a unital embedding of $R \subset R_1$ in another copy of the hyperfinite $\mathrm{II}_1$-factor such that for any inclusion of~$R_1$ as a subfactor of a larger hyperfinite $\mathrm{II}_1$ factor $R_2$, $\beta$ cannot be extended to an automorphism or an antiautomorphism of $R_2$. 
\end{Thm}
\begin{proof}
	We pick $\alpha \in \aut(R)$ as in Theorem \ref{Thm:free group},	 so that  the images of $\alpha$ and $\beta^{-1} \circ \alpha \circ \beta$ in $\out(R)$ generate a free group.
 Since $\alpha$ and all of its nonzero powers are outer, $R_1=R \rtimes_{\alpha} \Z$ is itself isomorphic to $R$ (being an injective $\mathrm{II}_1$-factor itself). We consider the standard embedding $R \subset R \rtimes_{\alpha} \Z$. Suppose $R \rtimes_{\alpha} \Z \subset R_2$ is a normal unital embedding into another copy of the hyperfinite $\mathrm{II}_1$-factor. We show that $\beta$ cannot extend to $R_2$. Suppose, for contradiction, that there exists $\widetilde{\beta} \in \ant(R_2) \cup \aut(R_2)$ such that $\widetilde{\beta}|_{R} = \beta$. Let $u$ be the canonical unitary in $R \rtimes_{\alpha} \Z$. Let $v = \widetilde{\beta}^{-1}(u)$ if $ \widetilde{\beta}$ is an automorphism, and $v = \widetilde{\beta}^{-1}(u^*)$ if~$ \widetilde{\beta}$ is an antiautomorphism. 
	
	We claim that for every $x \in R$ we have $vxv^* = \beta^{-1} \circ \alpha \circ \beta (x)$. To see this, note that if $ \widetilde{\beta}$ is an automorphism then
	\[
	vxv^* = 
		\widetilde{\beta}^{-1}(u) \beta^{-1}(\beta(x)) 	\widetilde{\beta}^{-1}(u^*) =
		\widetilde{\beta}^{-1}(u \beta(x) u^*) = \beta^{-1}(\alpha(\beta(x)))
	\]
	and if $ \widetilde{\beta}$ is an antiautomorphism then
	\[
		vxv^* = 
			\widetilde{\beta}^{-1}(u^*) \beta^{-1}(\beta(x)) 	\widetilde{\beta}^{-1}(u) =
			\widetilde{\beta}^{-1}(u \beta(x) u^*) = \beta^{-1}(\alpha(\beta(x)))
	\]	
Let $w \in F_2$ be a nontrivial word. Let $y = \pi_{u,v}(w)$ and  $\psi = \pi_{\alpha,\beta^{-1} \alpha \beta}(w)$  in $\aut(R)$. We denote the trace on $R_2$ by $\tilde{\tau}$. We claim that $\tilde{\tau}(y) =0$. 

	Note that $\Ad(y)$ leaves $R$ invariant, and for any $x \in R$, we have $\Ad(y)(x) =\psi(x)$. Since $ \psi $ and all of its nonzero powers are outer, using the Connes' Rokhlin-type theorem, 
	\cite[Chapter XVII, Lemma 2.3]{Takesaki-III}, where we pick $n=2$, for any $\eps>0$ there exist orthogonal projections $p_0,p_1 \in R$ such that $p_0 + p_1 = 1$,  $\|p_0\psi(p_0)\|_2 < \eps$, and  $\|p_1\psi(p_1)\|_2 < \eps$.  Thus, 
	\[
	\tilde{\tau}(y) = \tilde{\tau}(yp_0 \cdot p_0) + \tilde{\tau}(yp_1\cdot p_1) =   \tilde{\tau}(\psi(p_0) y p_0) + \tilde{\tau}(\psi(p_1) yp_1)
	\]
	Now, $|\tilde{\tau}(\psi(p_0) y p_0)| = |\tilde{\tau}(yp_0 \psi(p_0))| \leq \|y\| \|p_0 \psi(p_0)\|_2 < \eps$ and likewise $|\tilde{\tau}(\psi(p_0) y p_0)| < \eps$. Since $\eps$ was arbitrary, we have $\tilde{\tau}(y) = 0$. 
	
	We have shown that  	every nontrivial word $w$ satisfies $\tilde{\tau}(\pi_{u,v}(w)) = 0$. Therefore, Proposition~\ref{P:Wassermann} implies that $\cst(u,v) \cong \cst_r(F_2)$. However, by \cite[Corollary 4.2.4]{brown-memoirs}, the \cstar-algebra $\cst_r(F_2)$ does not embed into any finite, hyperfinite von Neumann algebra, which is a contradiction. 
\end{proof}

We move on to the fundamental group. 
For a \twoone{} factor $M$, $n\geq 1$, and a projection $p\in M_n(M)$, the isomorphism type of $pM_n(M)p$ depends only on the trace of $p$,  because in a \twoone{} factor projections of the same trace are unitarily equivalent.   
A representative of this isomorphism type is usually denoted $M^t$, where $t=\tau(p)$. The fundamental group $\cF(M)$ of $M$ is defined as $\{t\mid M^t\cong M\}$ (see e.g., \cite[\S 4.2]{anantharaman2017introduction}).

A small modification of Theorem \ref{P:extend} can be used for trace-scaling isomorphisms, as follows.
\begin{Thm}
\label{Prop:trace-scale-non-extend}
Let $p \in M_n(R)$ be a projection of trace $t \neq 1$. Suppose $\beta \colon R \to pM_n(R)p$ be an isomorphism. Then there exists 
a unital embedding of $R \subset R_1$ in another copy of the hyperfinite $\mathrm{II}_1$-factor such that for any inclusion of~$R_1$ as a subfactor of a larger hyperfinite $\mathrm{II}_1$ factor $R_2$, $\beta$ cannot be extended to an isomorphism $\widetilde{\beta} \colon R_2 \to pM_n(R_2)p$. 
\end{Thm}
\begin{proof}
Any such isomorphism $\beta$ arises from a trace-scaling automorphism $\gamma$ of $R \overline{\otimes} B(l^2(\N))$, restricted to $R \cong R \otimes q$, where $q$ is a minimal projection in $B(l^2(\N))$ (and corestricted to the image); here $\gamma(q)$ is a projection of trace~$t$. By \cite[Corollary 6]{Connes-1975}, any two such automorphisms are conjugate. In particular, if we identify $R \cong R \otimes R$, we can assume that $\gamma$ is of the form $\delta \otimes \gamma' \colon R \overline{\otimes} R \overline{\otimes} B(l^2(\N)) \to  R \overline{\otimes} R \overline{\otimes} B(l^2(\N))$, where $\delta \in \aut(R)$ is an outer automorphism and $\gamma' \colon  R \overline{\otimes} B(l^2(\N)) \to  R \overline{\otimes} B(l^2(\N))$ is a trace-scaling automorphism. Let $\tilde{p} = \gamma'(1 \otimes q)$, so that $p = 1 \otimes \tilde{p} = \gamma(1 \otimes 1 \otimes q)$.  Following the same argument as in the proof of Theorem \ref{P:extend}, we find $\alpha \in \aut(R)$ such that  $\left < \alpha , \delta^{-1} \circ \alpha \circ \delta \right > \cong F_2$. Now, let 
\[
R_1 = (R \overline{\otimes} R) \rtimes_{\alpha \otimes \id} \Z \cong (R \rtimes_{\alpha} \Z) \overline{\otimes} R
. 
\]
  Suppose $R_2 \supset R_1$ and $\beta$ extends to an isomorphism $\widetilde{\beta} \colon R_2 \to p R_2 \overline{\otimes} B(l^2(\N)) p$. Let $u$ be the canonical unitary in the crossed product $R \rtimes_{\alpha} \Z$ and  set $v = \widetilde{\beta}^{-1}(u \otimes \tilde{p})$. A  computation similar to that in the proof of Theorem \ref{P:extend} shows that for any $x \in R$, we have 
 \[
 v(x \otimes 1)v^* = \delta^{-1} \circ \alpha \circ \delta (x) \otimes 1 \in R \overline{\otimes} R \subset R_1 \subset R_2
 .
 \]
 The same considerations now show that $C^*(u \otimes 1,v) \cong C^*_r(F_2)$, which is a contradiction.
\end{proof}

In the next section we describe the recursive construction of the \twoone{} ~$M$ as in the conclusion of Theorem~\ref{Thm:Main} using the obstructions to extending $\beta$ provided by Theorem~\ref{P:extend} and Theorem~\ref{Prop:trace-scale-non-extend}. Notably, these obstructions are `irreversible' in the sense that $\beta$ cannot be extended to any further hyperfinite extension. This should be contrasted to the `fleeting' obstructions used in \cite{AkeWe:Consistency}, \cite{farah2016simple}, and \cite{vaccaro2017trace} where at each step of the construction one had to take care of all objects captured in the earlier stages of the construction.     

\section{The construction}

Following von Neumann,  an ordinal is identified with  the set of all smaller ordinals. By $\aleph_1$ we denote the first uncountable ordinal, identified with the first uncountable cardinal. 

As in \cite{farah2016simple} and \cite[\S 11.2]{Fa:STCstar}, our construction will utilize codes for metric structures\footnote{This is a technical term, following \cite{BYBHU}.} but the coding used here is somewhat simplified. 
Suppose~$d$ is a metric on an ordinal $\theta$ of diameter $2$ and $A$ is its metric completion. Let 
\[
\Code_d(A)=\{(\xi,\eta, q)\in \theta^2\times \Q_+\mid d(\xi,\eta)>q\}. 
\]
Since $\Code_d(A)$ uniquely determines the metric $d$ on $\xi$ and $A$ is isometric to the metric completion of this space, we consider $\Code_d(A)$ as a code for the metric space~$(A,d)$ (we will routinely omit $d$, when clear from the context). 
The set  $\Xm(\theta)$ of such codes is included in the power set of  $\theta^2\times \Q_+$.   
For every~$A$ coded in $\Xm(\theta)$, every 1-Lipschitz function $F\colon A^2\to [0,2]$ is coded by 
\[
\Code_F(A)=\{(\xi, \eta, q)\in \theta^2\times \Q_+\mid F(\xi,\eta)>q\}. 
\] 
Hence the pair $(A,F)$ is coded by a subset of $\theta^2\times \Q_+\sqcup \theta^2\times \Q_+$. 
Let $\XmR(\theta)$ denote the set of all such codes.

\begin{Lemma}\label{L:Struct} The sets of codes $\Xm(\xi)$ and $\XmR(\xi)$  satisfy the following for every infinite $\xi$. 
 \begin{enumerate}
% \item Each element of $\Xm(\xi)$ is a  code for a unique separable metric space of diameter 2 with a distinguished dense set indexed by $\xi$. Every separable metric space of diameter 2 is coded by an element of $\Xm(\xi)$.  
% \item Each  $\fA\in\XmR(\xi)$ is a code for a pair $(A,F)$ where $A$ is a separable metric space of diameter 2 and $R\colon A^2\to [0,2]$ is 1-Lipschitz.  Every such pair  is coded by some $\fA\in \XmR(\xi)$.  
\item \label{I.proj} There is a natural reduction from $\XmR(\xi)$ onto $\Xm(\xi)$, so that the reduct of a code for $(A,F)$ is a code for $A$ (with the same enumeration of the distinguished dense set). 
\item \label{I.exp1} If $A$ is coded by  $\fA\in\Xm(\xi)$ and $F\colon A^2\to [0,2]$ is 1-Lipschitz, 
then~$\fA$ has a unique expansion $\fA(F)$   in $\XmR(\xi)$ that codes $(A,F)$ (so that the reduct of $\fA(F)$ as in \eqref{I.proj} is $\fA$).  
\pushcounter\end{enumerate}
If $\xi<\eta$,  $\fA\in \Xm(\eta)$, and $\fA'\in \XmR(\eta)$,   then there are  unique $\fA\rs \xi\in \Xm(\xi)$ and $\fA'\rs \xi\in \XmR(\xi)$ with the following properties. 
\begin{enumerate}
\popcounter
\item \label{I.extend} If $A$ is coded by $\fA\in \Xm(\xi)$,  $\xi$ is countable, $A$ is a subspace of a separable metric space $B$ of diameter 2, then there is\footnote{By $\omega$ we denote the least limit ordinal, and therefore $\xi+\omega$ is the least limit ordinal greater than $\xi$.}  $\fB\in\Xm(\xi+\omega)$ such that $\fB\rs\xi=\fA$ and $\fB$ codes $B$.
\item \label{I.inductive} If  $S$ is set of ordinals and  $\fA_\xi\in \Xm(\xi)$, for $\xi\in S$, are such that $\fA_\eta\rs \xi=\fA_\xi$ for all $\xi<\eta$ in $S$, then there is a unique  $\fA\in \Xm\rs \sup S$ such that $\fA\rs \xi=\fA_\xi$ for all~$\xi\in S$. 
\item Statements analogous to \eqref{I.extend} and \eqref{I.inductive} hold when $\Xm$ is replaced with~$\XmR$. 
\end{enumerate}
\end{Lemma}

\begin{proof} The spaces $\Xm(\xi)$ and $\XmR(\xi)$ are instances of  $\Struct(\calL,\xi)$ for metric structures with a distinguished dense set indexed by $\xi$ as introduced in \cite[\S 7.1.2]{Fa:STCstar}, where $\calL$ is the single-sorted language with a single binary predicate symbol for $R$ whose modulus of uniform continuity is the identity function. 

To see that  \eqref{I.extend} holds, note that since $A$ is separable and the interval $[\xi,\xi+\omega)$ is infinite, one can extend the given enumeration of a dense subset of $A$ by $\xi$  to an enumeration of a dense subset of $B$ by~$\xi+\omega$.   
The proofs of the remaining clauses are even more straightforward. 
\end{proof}
		
		The unit ball $N_1$ of a \twoone{} factor $N$ with a separable predual with respect to a trace metric is a complete separable metric space of diameter 2, and if $\beta\colon N\to N$ is an automorphism, antiautomorphism, or an isomorphism onto a corner, then~$\beta$ can be coded by the distance function to its graph, denoted 
$
F_\beta \colon (N_1)^2\to [0,2]
$ and 
defined by 
\begin{equation} 
\label{Eq.Rbeta}
F_\beta(a,b)=\inf_{x\in N_1} \max(\|x-a\|_2, \|\beta(x)-b\|_2). 
\end{equation}
Clearly, $F_\beta$ is 1-Lisphitz.  

The following standard definitions can be found in \cite[\S III.6]{kunen2011set} or in  \cite[\S 6.2]{Fa:STCstar}. 
A subset $C$ of $\aleph_1$  is called closed and unbounded (\emph{club} for short) if $C\setminus \xi$ is nonempty for every $\xi<\aleph_1$ and for every countable $X\subset C$  we have $\sup(X)\in C$. A subset $S$ of $\aleph_1$  is stationary if it intersects every club nontrivially. We will not need the  exact statement of Jensen's $\doo$; it can be found e.g., in \cite[\S III.7.1]{kunen2011set} or \cite[\S 8.3.1]{Fa:STCstar}.

\begin{Prop}
\label{P:Diamond} Jensen's $\doo$ implies the following. 

There  exist $\sfS_\xi\in \XmR(\xi)$ for $\xi<\aleph_1$ such that for every $\fA\in \XmR(\aleph_1)$ the set 
$
\{\xi<\aleph_1\mid \fA\rs \xi=\sfS_\xi\}
$
is stationary. 

In particular, this statement is  relatively consistent with ZFC.  
\end{Prop}

\begin{proof} 
This is a consequence of a special case of \cite[Proposition~8.3.8]{Fa:STCstar} and the relative consistency of  $\doo$ with ZFC (\cite[\S III.7.13]{kunen2011set}). 
\end{proof}

We are now ready to prove Theorem \ref{Thm:Main}:

\begin{proof}[Proof of Theorem~\ref{Thm:Main}]  
	 We use $\diamondsuit_{\aleph_1}$ to construct a  \twoone{} factor that is not isomorphic to its opposite, has no outer automorphisms, has trivial fundamental group, is hyperfinite, is of density character $\aleph_1$, and such that the fundamental group of its ultrapower associated with any nonprincipal ultrafilter on $\N$ is equal to $(0,\infty)$.  
	 
Fix  $\sfS_\xi$, for $\xi<\aleph_1$ as guaranteed by Proposition~\ref{P:Diamond}. 
We will recursively build hyperfinite \twoone{} factors with separable predual $R_\xi$, for an infinite ordinal $\xi<\aleph_1$, and codes $\fA_\xi\in \Xm(\xi)$,  for limit $\xi\leq \aleph_1$, with the following properties ($F_\beta$ is as defined in \eqref{Eq.Rbeta}) 
\begin{enumerate}
\item If $\xi<\eta$ then $R_\xi$ is a subfactor of $R_\eta$  and for a limit ordinal~$\eta$ we have $R_\eta=\lim_{\xi<\eta} R_\xi$. 
\item If $\eta<\aleph_1$ is a limit ordinal, then a distinguished dense subset of the unit ball~$(R_\eta)_1$ of~$R_\eta$ in the trace metric is enumerated by $\eta$ and     $\fA_\eta\in \Xm(\eta)$ is the corresponding code for  $(R_\eta)_1$. 
\item For limit ordinals $\xi<\eta$ we have $\fA_\eta\rs \xi=\fA_\xi$.  
\item \label{MainStep} If  $\sfS_\xi=\fA_\xi(F_\beta)$  for some   $\beta$ which is an antiautomorphism,  an outer automorphism, or an isomorphism of $R_\xi$ onto a corner $p R_\xi p$ for some projection with $\tau(p)<1$, then $R_{\xi+1}$ is  $R_1$ as guaranteed by Theorem~\ref{P:extend} or \ref{Prop:trace-scale-non-extend}. 
\end{enumerate}
Starting from $R_\omega=R$,\footnote{As is standard in Set Theory, the letter $\omega$ denotes the first infinite countable ordinal; please note that in this paper $\omega$ does not stand for an ultrafilter and that $R_\omega$ does not stand for the central sequence algebra $R^\omega\cap R'$.}   the recursive construction proceeds as follows. Suppose that $\eta$ is the minimal ordinal such that $R_\eta$ hasn't been defined yet. 

If $\eta$ is a limit ordinal, let $R_\eta=\lim_{\xi<\eta} R_\xi$. If $\eta$ is also a limit of limits, then $\fA_\xi$ is defined for a cofinal set of $\xi<\eta$ and we let $\fA_\eta=\lim_{\xi<\eta} \fA_\xi$, as guaranteed by Lemma~\ref{L:Struct} \eqref{I.inductive}. 
Otherwise, let $\xi$ be the largest limit ordinal below $\eta$.   
Then $\eta=\xi+\omega$ and we let $\fA_{\eta}$ be a code for $R_\eta$ that extends~$\fA_\xi$ as guaranteed by Lemma~\ref{L:Struct} \eqref{I.extend}. 

Otherwise, $\eta$ is a successor ordinal. Fix $\xi$ such that 
 $\eta=\xi+1$. 
 
 Consider the case when   $\fA_\xi$ is defined  and there is  $\beta$ which is an antiautomorphism of $R_\xi$, an outer automorphism of $R_\xi$, or an isomorphism of~$R_\xi$ onto~$p R_\xi p$ for some projection $p$ of trace $< 1$ and $\sfS_\xi=\fA_\xi(F_\beta)$ (as in Lemma~\ref{L:Struct}  \eqref{I.exp1}).  Then let $R_\eta$ be as guaranteed by Theorem~\ref{P:extend} or Theorem~\ref{Prop:trace-scale-non-extend}, so that $\beta$ does not extend to any hyperfinite extension of $R_\eta$. This assures the requirement \eqref{MainStep} of the construction. 

In the case when  $\fA_\xi$ is not defined,  or $\fA_\xi$ is defined but $\sfS_\xi$ does not code a structure  $\fA_\xi(F_\beta)$ as in the previous case,  let $R_\eta=R_\xi$.

This describes the recursive construction. Let $M=\lim_{\xi<\aleph_1} R_\xi$ and let $\fA=\lim \fA_\eta$, as guaranteed by \eqref{I.inductive} of Lemma~\ref{L:Struct}. Then $M$ is  hyperfinite, as an inductive limit of hyperfinite \twoone{} factors and its predual has density character $\aleph_1$.

Suppose $\beta$ is an outer automorphism of $M$. 

We first claim that the set $\sfC_0=\{\xi<\aleph_1\mid \beta\rs R_\xi \in \aut(R_\xi)\}$
is a club. To see that, define a non-decreasing function $f: \aleph_1 \to \aleph_1$ by 
\[
f(\xi) = \min\{\eta < \aleph_1 \mid \beta[R_{\xi}] \cup \beta^{-1}[R_{\xi}] \subseteq R_{\eta}\} .
\]
The function $f$ is well-defined since each $R_{\xi}$ is separable in the $\|\cdot\|_2$-norm. The set $\sfC_0$ is the set of fixed points of $f$, and therefore is a club. (That the set of fixed points of such functions is a club is well-known; see for example \cite[Example 6.2.8(1)]{Fa:STCstar}.)

Next, we claim that  the set $\sfC_1=\{\xi\in \sfC_0\mid \beta\rs R_\xi$ is outer$\}$ 
also contains a club. The proof of this fact is essentially identical to the proof of the analogous statement for \cstar-algebras given in \cite[Proposition~7.3.9]{Fa:STCstar}, but we include the proof for reader's convenience. This proof uses the notion of a club in an arbitrary poset, defined  in \cite[Definitions 6.2.6]{Fa:STCstar} and  the poset of separable substructures of a metric structure,  defined in  \cite[Definition~7.1.8]{Fa:STCstar}. The supremum of an increasing sequence in the  poset of separable substructures is the closure of its union (typically strictly larger than the union), and this fact makes the proofs more involved than in the standard, discrete, case. 

To prove that $\sfC_1$ contains a club, first note that $\tilde\sfC_0=\{R_{\xi} \mid \xi \in \sfC_0\}$ is a club of separable substructures of~$M$ considered  as a metric structure (every tracial von Neumann algebra is naturally identified with a metric structure, see \cite[\S 2.3.2]{FaHaSh:Model2}).    
Assume for contradiction that $\sfC_1$ does not contain a club, so $\sfC_0 \setminus \sfC_1$ is stationary. Since the function $\xi\mapsto R_\xi$ is an order-isomorphism between $\sfC_0$ and $\tilde \sfC_0$, the set  $\{R_\xi\mid \xi\in \sfC_0\setminus \sfC_1\}$ is stationary as well. For each $\xi \in \sfC_0 \setminus \sfC_1$, by our assumption, $\beta|_{R_{\xi}}$ is inner, so we can choose a unitary $u_{\xi} \in R_{\xi}$ such that $\beta|_{R_{\xi}} = \Ad(u_{\xi})$. The function $R_{\xi} \mapsto u_{\xi}$ is regressive (this simply means that $u_{\xi} \in R_{\xi}$, see \cite[Definition 7.2.1]{Fa:STCstar}). By \cite[Proposition 7.2.9]{Fa:STCstar}, there exists $u \in M$ such that for any $\eps>0$, the set
\[
\{R_{\xi} \mid \|u - u_{\xi} \|_2 < \eps\}
\]
is stationary. From here it follows that $\beta = \Ad(u)$. To see this, note that for any $a$ in the unit ball of~$M$ and for any $\eps>0$, pick $\xi$ such that $a \in R_{\xi}$ and $R_{\xi}$ is in the stationary set above, so $\|\Ad(u)(a) - \beta(a)\|_2<2\eps$. Since $\eps>0$ is arbitrary, it follows that $\beta = \Ad(u)$. This
 is a contradiction. Therefore, the set $\sfC_1$ contains a club, as claimed.

By the choice of $\sfS_\xi$, since the limit ordinals form a club, there is a limit ordinal $\xi\in \sfC_1$  such that $\fA_\xi ( F_{\beta\rs R_\xi})=\sfS_\xi$. 
By case \eqref{MainStep} of the construction, the subfactor $R_{\xi+1}$ of $M$  has the property that for any larger hyperfinite~\twoone{} factor~$N$, no $\tilde\beta\in \aut(N)$ extends $\beta\rs R_\xi$. However, $M$ and $\beta$ have this property; contradiction.  

Now suppose that $\beta$ is an antiautomorphism. As before,  the set 
$
\sfC=\{\xi<\aleph_1\mid \beta\rs R_\xi \in \ant(R_\xi)\}
$
is a club. Thus there exists $\xi\in \sfC$ such that $\sfS_\xi=\fA_{\xi}(F_{\beta\rs R_\xi})$.
By case \eqref{MainStep} of the construction, the subfactor $R_{\xi+1}$ of $M$  has the property that for any larger hyperfinite~\twoone{} factor~$N$, no $\tilde\beta\in \ant(N)$ extends $\beta\rs R_\xi$; contradiction.

Finally,  suppose that  $\cF(M)\neq \{1\}$. Since $\cF(M)$ is a multiplicative group,  there exists $0<t<1$ and an isomorphism $\beta\colon M\to pMp$ for some $p$ with $\tau(p)<1$. Since any two projections with the same trace are unitarily equivalent, we may assume without loss of generality that $p \in R_{\omega}$, so that $p \in R_{\xi}$ for all infinite ordinals $\xi$ we consider.
By the same argument as above, the set 
\[
\sfC=\{\xi<\aleph_1\mid \beta\rs R_\xi\text{ is an isomorphism onto } p R_\xi p\}
\]
is a club. 
Thus there is a limit ordinal~$\xi\in \sfC$  such that $\fA_\xi (F_{\beta\rs R_\xi})=\sfS_\xi$. 
By case \eqref{MainStep} of the construction, the subfactor $R_{\xi+1}$ of $M$  has the property that for any larger hyperfinite~\twoone{} factor~$N$, no $\tilde\beta$ extends $\beta\rs R_\xi$; contradiction.

It remains to prove $\cF(M^\cV)=(0,\infty)$ for any nonprincipal ultrafilter $\cV$ on~$\N$. 
 Since $\cF(R)=(0,\infty)$ and an isomorphism of $R$ with its corner $R^t$ extends to an isomorphism of $R^\cV$ with  $(R^t)^\cV\cong (R^\cV)^t$, it will suffice to prove that $M^\cV\cong R^\cV$.
This is a consequence of a standard model-theoretic fact. By \cite[Corollary~16.4.2]{Fa:STCstar}, we know that $M^\cV$ and $R^\cV$ are countably saturated (\cite[Definition 16.1.5]{Fa:STCstar}). Since $\doo$ implies the Continuum Hypothesis, both $M^\cV$ and $R^\cV$ have density character $\aleph_1$, and therefore are in fact saturated.
By construction,  $M$ is the inductive limit of   $R_\xi$, for $\xi<\aleph_1$. Every $R_\xi$ is isomorphic to $R$, and for a countable limit ordinal $\eta$ we have that  $R_\eta$ is the  $\|\cdot\|_2$-closure of $\bigcup_{\xi<\eta}R_\xi$. By the Downwards L\"owenheim--Skolem Theorem (e.g., \cite[Theorem~7.1.9]{Fa:STCstar}), some $R_\xi$ is an elementary submodel of $M$ in the language of tracial von Neumann algebras.\footnote{As a matter of fact, \emph{every} $R_\xi$ is an elementary submodel of $M$ but we don't have to use this result.} Therefore $R_\xi^\cV\cong M^\cV$ (e.g., 	\cite[Corollary~16.6.5]{Fa:STCstar}). Since $R\cong R_\xi$,  this concludes the proof. 
\end{proof}

\begin{Rmk}
We do not know whether Theorem~\ref{Thm:Main}, or the main result of any of  \cite{AkeWe:Consistency}, \cite{farah2016simple}, or \cite{vaccaro2017trace}, can be proved in ZFC alone, or from the Continuum Hypothesis. In \cite{calderon2019can} it was shown that the conclusions of the main results from \cite{AkeWe:Consistency}, \cite{farah2016simple}, and \cite{vaccaro2017trace} all hold in many models of the Continuum Hypothesis in which $\doo$ fails. 

We don't know whether in ZFC one can prove that $\cF(M^\cV)=(0,\infty)$ for every hyperfinite \twoone{} factor $M$ with predual of density character $\aleph_1$ and every ultrafilter $\cV$ on $\N$ (although this is certainly true for $R$).  We conjecture that this is not necessarily the case and that the results of  \cite{shelah1992vive}, showing that ultrapowers of countable, elementarily equivalent,  structures associated with nonprincipal ultrafilters on $\N$ need not be isomorphic,  may be relevant.  
\end{Rmk}

\bibliographystyle{alpha}
\bibliography{nonseparable-type-II}

\begin{thebibliography}{BYBHU08}

\bibitem[AP17]{anantharaman2017introduction}
C.~Anantharaman and S.~Popa.
\newblock {\em An introduction to {II$_1$} factors}.
\newblock available at https://idpoisson.fr/anantharaman/publications/IIun.pdf,
  2017.

\bibitem[AW04]{AkeWe:Consistency}
C.~Akemann and N.~Weaver.
\newblock Consistency of a counterexample to {N}aimark's problem.
\newblock {\em Proc. Natl. Acad. Sci. USA}, 101(20):7522--7525, 2004.

\bibitem[Bro06]{brown-memoirs}
N.~P. Brown.
\newblock Invariant means and finite representation theory of {$C^*$}-algebras.
\newblock {\em Mem. Amer. Math. Soc.}, 184(865):viii+105, 2006.

\bibitem[BYBHU08]{BYBHU}
I.~Ben~Yaacov, A.~Berenstein, C.W. Henson, and A.~Usvyatsov.
\newblock Model theory for metric structures.
\newblock In Z.~Chatzidakis et~al., editors, {\em Model Theory with
  Applications to Algebra and Analysis, Vol. II}, number 350 in London Math.
  Soc. Lecture Notes Series, pages 315--427. London Math. Soc., 2008.

\bibitem[CF20]{calderon2019can}
D.~Calder\'on and I.~Farah.
\newblock Can you take {A}kemann--{W}eaver's diamond away?
\newblock preprint, 2020.

\bibitem[Con75a]{Connes-not-anti-isomorphic}
A.~Connes.
\newblock A factor not anti-isomorphic to itself.
\newblock {\em Ann. Math. (2)}, 101:536--554, 1975.

\bibitem[Con75b]{Connes-1975}
A.~Connes.
\newblock Outer conjugacy classes of automorphisms of factors.
\newblock {\em Ann. Sci. \'{E}cole Norm. Sup. (4)}, 8(3):383--419, 1975.

\bibitem[Con75c]{Connes-comptes-rendus}
A.~Connes.
\newblock Sur la classification des facteurs de type {${\rm II}$}.
\newblock {\em C. R. Acad. Sci. Paris S\'{e}r. A-B}, 281(1):Aii, A13--A15,
  1975.

\bibitem[Far19]{Fa:STCstar}
I.~Farah.
\newblock {\em Combinatorial Set Theory and \cstar-algebras}.
\newblock Springer Monographs in Mathematics. Springer, 2019.

\bibitem[FH17]{farah2016simple}
I.~Farah and I.~Hirshberg.
\newblock Simple nuclear \cstar-algebras not isomorphic to their opposites.
\newblock {\em Proc. Natl. Acad. Sci. USA}, 114(24):6244--6249, 2017.

\bibitem[FHS14]{FaHaSh:Model2}
I.~Farah, B.~Hart, and D.~Sherman.
\newblock Model theory of operator algebras {II}: Model theory.
\newblock {\em Israel J. Math.}, 201:477--505, 2014.

\bibitem[FK15]{FaKa:NonseparableII}
I.~Farah and T.~Katsura.
\newblock Nonseparable {UHF} algebras {II}: Classification.
\newblock {\em Math. Scand.}, 117(1):105--125, 2015.

\bibitem[Gio83]{Giordano}
T.~Giordano.
\newblock Antiautomorphismes involutifs des facteurs de von {N}eumann
  injectifs. {I}.
\newblock {\em J. Operator Theory}, 10(2):251--287, 1983.

\bibitem[IPP08]{Ioana-Peterson-Popa}
A.~Ioana, J.~Peterson, and S.~Popa.
\newblock Amalgamated free products of weakly rigid factors and calculation of
  their symmetry groups.
\newblock {\em Acta Math.}, 200(1):85--153, 2008.

\bibitem[Kis81]{kishimoto-1981}
A.~Kishimoto.
\newblock Outer automorphisms and reduced crossed products of simple {$C^{\ast}
  $}-algebras.
\newblock {\em Comm. Math. Phys.}, 81(3):429--435, 1981.

\bibitem[KOS03]{KiOzSa}
A.~Kishimoto, N.~Ozawa, and S.~Sakai.
\newblock Homogeneity of the pure state space of a separable {$C\sp
  *$}-algebra.
\newblock {\em Canad. Math. Bull.}, 46(3):365--372, 2003.

\bibitem[Kun11]{kunen2011set}
K.~Kunen.
\newblock {\em Set theory}, volume~34 of {\em Studies in Logic}.
\newblock College Publications, London, 2011.

\bibitem[She92]{shelah1992vive}
S.~Shelah.
\newblock Vive la diff{\'e}rence {I}: {N}onisomorphism of ultrapowers of
  countable models.
\newblock In {\em Set theory of the continuum}, pages 357--405. Springer, 1992.

\bibitem[Tak03]{Takesaki-III}
M.~Takesaki.
\newblock {\em Theory of operator algebras. {III}}, volume 127 of {\em
  Encyclopaedia of Mathematical Sciences}.
\newblock Springer-Verlag, Berlin, 2003.
\newblock Operator Algebras and Non-commutative Geometry, 8.

\bibitem[Vac18]{vaccaro2017trace}
A.~Vaccaro.
\newblock Trace spaces of counterexamples to {N}aimark's problem.
\newblock {\em J. Funct. Anal.}, 275(10):2794--2816, 2018.

\bibitem[Was76]{wassermann1976tensor}
S.~Wassermann.
\newblock On tensor products of certain group \cstar-algebras.
\newblock {\em J. Funct. Anal.}, 23(3):239--254, 1976.

\bibitem[Wid57]{Wid:Nonisomorphic}
H.~Widom.
\newblock Nonisomorphic approximately finite factors.
\newblock {\em Proc. Amer. Math. Soc.}, 8:537--540, 1957.

\end{thebibliography}
\end{document}